\documentclass[12pt]{article}
\usepackage{amsmath,amsfonts,amssymb,amsthm,amscd}

\usepackage[all]{xy} 
\usepackage{graphicx}

\title{On Levi subgroups and the Levi decomposition for groups definable in $o$-minimal structures} 
\date{November 7, 2011}
\author{Annalisa Conversano\\ \and Anand Pillay\thanks{Supported by EPSRC grant EP/I002294/1}\\University of Leeds}
 
\newtheorem{Theorem}{Theorem}[section]
\newtheorem{Proposition}[Theorem]{Proposition}
\newtheorem{Definition}[Theorem]{Definition} 
\newtheorem{Remark}[Theorem]{Remark}
\newtheorem{Lemma}[Theorem]{Lemma}

\newtheorem{Fact}[Theorem]{Fact}

\newcommand{\R}{\mathbb R}  
  
\newcommand{\Z}{\mathbb Z}

\renewcommand{\leq}{\leqslant}

\begin{document}
\maketitle

\begin{abstract} We study analogues of the notions from Lie theory of Levi subgroup and Levi decomposition, in the case of groups $G$ definable in an $o$-minimal expansion of a real closed field. With suitable definitions, we prove that $G$ has a unique maximal {\em ind-definable semisimple} subgroup $S$, up to conjugacy, and that $G = R\cdot S$ where $R$ is the solvable radical of $G$. We also prove that any semisimple subalgebra of the Lie algebra of $G$ corresponds to a unique ind-definable semisimple subgroup of $G$. 
\end{abstract}

\section{Introduction and preliminaries}

The ``Levi-Mal'cev" theorem  sometimes refers to  Lie algebras (over any field of characteristic $0$) and sometimes to Lie groups. 
For Lie algebras $L$ it says that $L$ is the semidirect product of a solvable ideal $\mathfrak{r}$ and a semisimple subalgebra $\mathfrak s$ (with certain uniqueness properties) and, as such, is valid for Lie algebras over real closed fields. $\mathfrak{s}$ is sometimes called a Levi factor of $L$. For connected Lie groups $G$ it says that $G$ has a unique, up to conjugacy,  maximal connected semisimple Lie subgroup $S$, and for any such $S$, $G = R\cdot S$ where $R$ is the solvable radical (maximal connected solvable normal subgroup). And of course $R\cap S$ is $0$-dimensional. $S$ need not be closed, but when $G$ is simply connected, $S$ {\em is} closed and $R\cap S = \{1\}$ (so $G$ is the semidirect product of $R$ and $S$). $S$ is sometimes called a Levi subgroup of $G$. See Theorems 3.14.1, 3.14.2 and 3.18.13 of \cite{Varadarajan} for example. We also refer to the latter book for basic facts and definitions concerning Lie algebras and Lie groups.\\

In this paper we are concerned with $G$  a (definably connected) group definable in an $o$-minimal expansion $M$ of a real closed field $K$, so this goes outside the Lie group context unless $K = \R$. We are interested not only in the existence of a ``Levi" subgroup and decomposition of the group $G$ but also in definability properties. Even in the case where $M = 
(\R,+,\cdot)$ and so $G$ is a {\em Nash group}  (semialgebraic Lie group), this is a nontrivial issue and $S$ need not be semialgebraic, as pointed out in the first author's thesis (see also Example 2.10 of \cite{CCI}).  
In the general situation $G$ will have a Lie algebra $L(G)$ (over the relevant real closed field $K$) which has its own Levi decomposition as a sum of a solvable ideal and a semisimple algebra, so the issue is what kind of subgroup of $G$, if any, corresponds to the semisimple subalgebra, and also to what extent it is unique. 

We will be forced into the category of ``$ind$-definable" subgroups, i.e. defined by a possibly infinite, but countable, disjunction of formulas. We will give an appropriate (strong) definition of an {\em  ind-definable semisimple subgroup} of $G$, and in analogy with the classical case, prove the existence and uniqueness up to conjugacy of a maximal such ind-definable semisimple subgroup $S$ of $G$, as well as that $G = R\cdot S$ where $R$ is the solvable radical of $G$. Definability of $S$ in this general context corresponds more or less to $S$ being a closed subgroup of $G$ in the classical context. As remarked above, the  first example of nondefinability of $S$ was given in \cite{CCI}. We will also give a number of situations where $S$ is definable.   \\

We now aim towards stating formally the main result.
We assume $M$ to be an $o$-minimal expansion of a real closed field, and $G$ to be a definably connected definable group in $M$. $G$ has a unique maximal definably connected solvable subgroup which we call $R$. We denote the quotient $G/R$ by $P$, a definable (equivalently interpretable) group which is definably connected and {\em semisimple} in the sense that it has no infinite normal solvable (equivalently abelian) subgroups, or equivalently the quotient of $P$ by its finite centre is a direct product of finitely many definably simple (noncommutative) definable groups. We sometimes call $P$ the semisimple part of $G$. \\

A crucial notion in this paper is that of an {\em ind-definable semisimple} group (or subgroup of a given group) $S$. The meaning is that $S$ is ind-definable, (locally) definably connected, has a ``discrete" centre $Z(S)$ and the quotient is a definable semisimple group. The notions will be explained in section 2.

\begin{Theorem} $G$ has a maximal ind-definable semisimple subgroup $S$, unique up to conjugacy in $G$.
Moreover 
\newline 
(i) $G = R\cdot S$,
\newline
(ii) The centre of $S$, $Z(S)$,  is finitely generated and contains $R\cap S$. 
\end{Theorem}

It will also follow from material in section 2 that if $\pi:G\to P$ is the canonical surjective homomorphism, then the (surjective) homomorphism from $S$ to $P$ induced by $\pi$ is a quotient of the $o$-minimal universal cover of $P$. 
We will call $S$ as in Theorem 1.1 an {\em ind-definable Levi subgroup} of $G$, and the decomposition of $G$ given by Theorem 1.1 the {\em ind-definable Levi decomposition} of $G$. When $G$ is a definable real Lie group this decomposition  coincides with the usual Levi decomposition of $G$ referred to earlier. Note that by uniqueness of the (ind-definable) Levi subgroup up to conjugacy, some Levi subgroup will  be definable iff all are. 
When $K = \R$ the examples of nondefinability of the Levi subgroup, given in \cite{CCI}, \cite{CCII} and \cite{Conversano-thesis},  come from encoding the universal cover of $P$ as an ind-definable but non definable subgroup of $G$, and for this to be possible $P$ has to have infinite ``fundamental group".\\

Our methods will also yield:
\begin{Theorem} Let ${\mathfrak s}$ be a semisimple Lie subalgebra of $L(G)$. Then there is a unique ind-definable semisimple subgroup $S$ of $G$ such that $\mathfrak{s} = L(S)$.
\end{Theorem}
 
\vspace{2mm}
\noindent
In section 2 we discuss ind-definable groups, semisimplicity,  and universal covers. In sections 3 and 4 we will prove Theorem 1.1 (and Theorem 1.2) in some special cases, and then combine these in section 5 to give the proofs in general. At the end of section 5 we will list a number of hypotheses which imply definability of the Levi subgroups. In the remainder of this introduction we will recall some basic facts and notions. \\

Usually $M$ denotes an $o$-minimal expansion of a real closed field $K$ and $G$ a group definable in $M$. For various reasons, especially when dealing with ind-definable objects, we should bear in mind a saturated elementary extension ${\bar M}$ of $M$. We refer to earlier papers such as \cite{PPSI} for an account of the  general theory of definable sets and definable groups in $M$, as well as the existence and properties of tangent spaces and Lie algebras of definable groups. But we repeat that for any $k$ a definable group can be equipped with a (essentially unique) definable $C^{k}$-manifold structure over $K$ with respect to which the group operation is $C^{k}$. Likewise for definable homogeneous spaces. Definable connectedness of a definable group has two equivalent descriptions; no proper definable subgroup of finite index, and no proper open definable subgroup with respect to the topological structure referred to above. Definability means with parameters unless we say!
  otherwise.

\begin{Definition}  $G$ is {\em semisimple} if $G$ is definably connected and has no infinite normal abelian (definable) subgroup.
\end{Definition}

\begin{Remark} Assume $G$ definably connected. $G$ is semisimple if and only if $Z(G)$ is finite and $G/Z(G)$ is a direct product of finitely many definably simple, noncommutative, definable groups.
\end{Remark}

We now list some basic facts, from \cite{HPP}, \cite{OPP}, \cite{PPSI}, \cite{PPSIII},  which we will use:

\begin{Fact}
(i)  Assume $G$ is definably connected. Then $G$ has a unique maximal definable definably connected normal solvable subgroup $R$ and $G/R$ is semisimple.
\newline
(ii) If $G$ is semisimple then $G$ is perfect (i.e. $G$ equals its commutator subgroup $[G,G]$), and moreover for some $r$ every element of $G$ is a product of at most $r$ commutators.
\newline
(iii) If $G$ is definably connected, then $G/Z(G)$ is {\em linear}, namely definably embeds in some $GL_{n}(K)$.
\newline
(iv) Let $G$ be definably connected. Then $G$ is semisimple iff $L(G)$ is semisimple.
\newline
(v) If ${\mathfrak s}$ is a semisimple Lie subalgebra of $\mathfrak{gl}_{n}(K)$, then there is a (unique) definably connected definable subgroup $S$ of $GL_{n}(K)$ such that $\mathfrak{s} = L(S)$. Moreover $S$ is semialgebraic  (and semisimple by (iv)). 
\newline
(vi) If $G$ is definable, semisimple and centreless, then $G$ is definably isomorphic to a semialgebraic subgroup of some $GL_{n}(K)$ which is defined over $\R$ (in fact over $\Z$). 
\end{Fact}

\vspace{5mm}
\noindent
This paper is closely related to our earlier papers \cite{CCI}, \cite{CCII}, to \cite{HPP}, and to themes in the first author's thesis \cite{Conversano-thesis}.  The first author would like to thank her advisor Alessandro Berarducci, as well as Ya'acov Peterzil.

\section{Ind-definability, semisimplicity, and universal covers}

The expressions {\em ind-definable}, {\em $\vee$-definable}, and {\em locally definable} are more or less synonymous, and refer to definability by a possibly infinite disjunction of first order formulas. There is a considerable literature on ind-definability and the ``category" of ind-definable sets. See for example the detailed treatment in section 2.2 of \cite{HL}. Likewise there is a lot written on ind-definable spaces and groups in the $o$-minimal setting, especially in the context of universal covers and fundamental groups. See for example \cite{EdmundoI} and \cite{EdmundoII}. So we refer to these other sources for more details and  restrict ourselves here to fixing notation suitable for the purposes of this paper. 

We start with $T$ an arbitrary complete theory in a countable language $L$ say and ${\bar M}$ a saturated model of $T$. A definable set means a definable set in ${\bar M}$ unless we say otherwise. $M$ denotes a small elementary substructure. 
\begin{Definition} (i) By an (abstract) $ind$-definable set $X$ we mean a countable collection $(X_{i}:i<\omega)$ of definable sets together with definable injections $f_{i}:X_{i}\to X_{i+1}$ for each $i$, where we identify $X$ with the directed union (via the $f_{i}$) of the $X_{i}$. By a definable subset of $X$ we mean a definable subset of some $X_{i}$ (in the obvious sense). We say $X$ is defined over $M$ if the $X_{i}$ and $f_{i}$ are, in which case we also have $X(M)$.
\newline
(ii) An $ind$-definable group is an ind-definable set as in (i) such that $X$ has a group operation which is definable, namely the restriction to each $X_{i}\times X_{j}$ is definable, hence with image in some $X_{k}$. 
\newline
(iii) An ind-definable set $X$ as in (i) is called a concrete ind-definable set if for some definable set $Y$, all the $X_{i}$ are subsets of $Y$ and 
each $f_{i}$ is the identity map restricted to $X_{i}$, namely $X_{i}\subseteq X_{i+1}$ for all $i$ whereby $X$ is simply $\bigcup_{i}X_{i}$, an ind-definable subset of $Y$. 
\end{Definition}

\begin{Remark} (i) If $X$ is an ind-definable subset of the definable set $Y$ as in (iii) above, and $Z$ is a definable subset of $Y$ contained in $X$, then by compactness $Z$ is contained in some $X_{i}$. Hence the notion of a {\em definable subset} of the abstract ind-definable set $X$ is consistent with the natural notion when $X$ is concrete.
\newline
(ii) There are obvious notions of a function between (abstract) ind-definable sets being definable (or we should say ind-definable). Note in particular that if $X$ is ind-definable, $Y$ is definable and $f:X\to Y$ is definable and surjective then already the restriction of $f$ to some $X_{i}$ is surjective (by compactness). 
\end{Remark}
\vspace{2mm}
\noindent

We can formulate some basic notions such as definable connectedness for groups at this level of generality.
\begin{Definition} (i) Let $X$ be an ind-definable set and $Y$ a subset of $X$. We will say that $Y$ is discrete if for any definable subset $Z$ of $X$, $Z\cap Y$ is finite.
\newline
(ii) Let $G$ be an ind-definable group. We will call $G$ {\em definably connected} if $G$ has {\em no} proper subgroup $H$ with the properties: for each definable subset $Z$ of $G$, $Z\cap H$ is definable and $Z$ meets only finitely many distinct cosets of $H$ in $G$. 
\end{Definition} 

Maybe we should rather use the expression ``locally definably connected" in (ii) above, but we leave it as is. In any case when $X$ is a definable set  ($G$ a definable group) the above notions reduce to $Y$ is finite ($G$ has no proper definable subgroup of finite index). Let us state for the record:

\begin{Lemma} Let $G$ be a definably connected ind-definable group. Then any discrete normal subgroup of $G$ is central. 
\end{Lemma}
\begin{proof} Let $N$ be a discrete normal ind-definable subgroup of $G$. Then $G$ acts on $N$ by conjugation. Let $n\in N$ and let $H$ be $C_{G}(n)$ which clearly meets each definable subset of $G$ in a definable set. Let $Z$ be a definable subset of $G$. Then $\{gng^{-1}: g\in Z\}$ is a definable subset of $N$, so finite as $N$ is discrete. So only finitely many distinct cosets of $H$ in $G$ meet $Z$. As this is true for all definable $Z$ and $G$ is definably connected we see that $H = G$, i.e. $n$ is central in $G$.
\end{proof}

We now specialize to the $o$-minimal case, i.e. $T$ is an $o$-minimal expansion of $RCF$. We will only work with concrete ind-definable sets. When  $X = G$ is a (concrete) ind-definable group, then by \cite{EdmundoI} $X$ can be definably equipped with a topology such that the group operation is continuous (as in the case for definable groups), and in fact $C^{k}$ for arbitrarily large $k$.  Definable connectedness as defined above has a ``topological" interpretation. Also $G$ has a well-defined Lie algebra (over the ambient real closed field). Here is our main definition (which agrees with the usual one when $G$ is definable).

\begin{Definition} 
We will call $G$ {\em ind-definable semisimple} if $G$ is ind-definable and definably connected, $Z(G)$ is discrete, and $G/Z(G)$ is definable and semisimple (namely there is a definable semisimple group $D$ and a definable surjective homomorphism from $G$ to $D$ with kernel $Z(G)$, and note that $D$ will be centreless). 
\end{Definition}

\begin{Remark} An equivalent definition is: $G$ is ind-definable, definably connected, and there is a definable surjective homomorphism $\pi$ from $G$ to a definable (not necessarily centreless) semisimple group $D$ such that $ker(\pi)$ is discrete. 
\end{Remark}

\begin{Lemma} An ind-definable semisimple group is perfect.
\end{Lemma}
\begin{proof} Let $G$ be our ind-definable semisimple group, and $\pi:G\to D$ definable with $D$ definable semisimple. Let $G_{1} = [G,G]$. We want to argue that (i) the intersection of $G_{1}$ with any definable subset $Z$ of $G$ is definable and that moreover (ii) $Z$ intersects only finitely many distinct cosets of $G_{1}$ in $G$. Definable connectedness of $G_{1}$ will then imply that $G_{1} = G$. 
\newline
We first prove (i). It suffices to show that for arbitrarily large definable subsets $Y$ of $G$, $G_{1}\cap Y$ is definable. As $D = [D,D]_{r}$ (the collection of products of $r$ commutators), we may assume, by enlarging $Y$ that $Y\cap [G,G]$ maps onto $D$ under $\pi$. Now clearly $Y\cap [G,G]$ is ind-definable.
\newline
{\em Claim.} $Y\setminus [G,G]$ is ind-definable.
\newline
{\em Proof of claim.} Let $y\in Y\setminus [G,G]$, and let $\pi(y) = d\in D$. By our assumption above there is $x\in Y\cap [G,G]$ such that $\pi(x) = d$. Hence $x^{-1}y\in ker(\pi) = Z(G)$. Note that $x^{-1}y\in Y^{-1}\cdot Y$ a definable subset of $G$. By definition of $G$ being ind-definable semisimple, $Z(G)\cap (Y\cdot Y^{-1})$ is finite. Hence $Y\setminus [G,G]$ equals the union of translates $c\cdot(Y\cap[G,G])$ for $c$ ranging over the
(finite) set of elements of $Z(G) \cap (Y\cdot Y^{-1})$ which are not in $[G,G]$. This proves the claim.

\vspace{2mm}
\noindent
By the claim and compactness, $Y\cap [G,G]$ is definable. We have proved (i). The proof of the claim shows (ii). So the lemma is proved.
\end{proof}

\begin{Lemma} Let $G$ be an ind-definable semisimple group. Then $L(G)$ is semisimple.
\end{Lemma}
\begin{proof} Let $\pi: G\to D$ be the canonical surjective homomorphism to a definable semisimple group. As $ker(\pi)$ is discrete, $\pi$ induces an isomorphism between $L(G)$ and $L(D)$ and the latter is semisimple by 1.5(iv). 

\end{proof}

As remarked earlier there is a body of work on $o$-minimal universal covers, which it will be convenient to refer to (although we could use other methods, such as in section 5 of \cite{CCII}).  The content of Theorem 1.4 of \cite{Edmundo-Pantelis} is: 

\begin{Fact} Let $G$ be a definable, definably connected group. Then 
\newline 
(i) The family $Cov(G) = \{f:H\to G$, where $H$ is ind-definable, definably connected, 
$f$ is surjective and definable with discrete kernel\} has a universal object, namely some $\pi:{\tilde G} \to G$ in $Cov(G)$ such that for any $f:H\to G$ in $Cov(G)$ there is a (unique) surjective definable homomorphism $h:{\tilde G} \to H$ such that  $h\circ f = \pi$. 
\newline
(ii) Moreover the kernel of $\pi:{\tilde G} \to G$ is finitely generated.
\end{Fact}

$\pi:{\tilde G} \to G$ is what is known as the ``$o$-minimal universal cover of $G$" and it is proved in \cite{Edmundo-Pantelis} that $ker(\pi)$ coincides with the ``$o$-minimal fundamental group" $\pi_{1}(G)$ of $G$ given in terms of definable paths and homotopies in \cite{B-O}.  Also if $G$ is definable over $M$ so is $\pi:{\tilde G}\to G$. 

\vspace{2mm}
\noindent
Although not required for the purposes of this paper, one would also expect $\pi:{\tilde G} \to G$ to have the additional property: 
\newline
For any (locally) definable central extension $f: H \to G$ of $G$ there is a (unique) definable (but not necessarily surjective) homomorphism $h:{\tilde G} \to H$ such that $\pi = h\circ f$. 

\vspace{2mm}
\noindent
\begin{Remark} If $G$ is an ind-definable semisimple group and $f:G\to G/Z(G)= H$ is the canonical surjective homomorphism from $G$ to a definable semisimple group, then by Fact 2.9(i)  $f$ is a quotient of the $o$-minimal universal cover $\pi:{\tilde H} \to H$, and by Fact 2.9(ii), $Z(G)$ is finitely generated. 
\end{Remark}

We now briefly recall the relation of the $o$-minimal universal covers to the classical universal covers of connected Lie groups. Let us suppose that $G$ is a definably connected definable group (identified with its group of ${\bar M}$-points). Let $L^{-}$ be a sublanguage of the language $L(T)$ of $T$ including the language of ordered fields, such that for some copy $\R$ of the reals living inside $K$, $\R$ with its induced $L^{-}$-structure is an elementary substructure of ${\bar M}|L^{-}$. Let us suppose also that $G$ is definable in $L^{-}$ with parameters from $\R$. Then $G(\R)$ is a connected Lie group. Moreover the $o$-minimal universal cover $\tilde G$ of $G$ is (ind)-definable in $L^{-}$ over $\R$ whereby ${\tilde G}(\R)$ makes sense as a topological (in fact Lie) group. Then Theorem 8.5 and its proof from \cite{HPP} says that ${\tilde G}(\R)$ is the classical universal cover of $G(\R)$. Moreover the kernel of ${\tilde G} \to G$ coincides with the kernel of ${\tilde!
  G}(\R) \to G(\R)$, which is the fundamental group of
the Lie group $G(\R)$. 

This applies in particular to the case when $G$ is semisimple: By \cite{HPP}, $G$ is definably isomorphic in ${\bar M}$ to a group definable in the ordered field language over $\R$. So we may assume $G$ to be already definable in the ordered field language over $\R$. So the $o$-minimal fundamental group of $G$ coincides with the fundamental group of the semisimple real Lie group $G(\R)$. 

\vspace{5mm}
\noindent
In the rest of the paper, the model $M$ will be an arbitrary model of an $o$-minimal expansion of $RCF$. When we speak of of an ind-definable set (group) in $M$ we mean $X(M)$ ($G(M)$) for $X$ ($G$) an ind-definable set (group) in ${\bar M}$ which is defined over $M$.


\section{Central extensions of definable semisimple groups}
Here we prove Theorem 1.1 when $R$ is central in $G$, hence $R = Z(G)^{0}$. In fact $S$ will be turn out to be the commutator subgroup $[G,G]$, but one has to check the various properties claimed of $S$. 

We start with a trivial fact about {\bf abstract groups}, which we give a proof of, for completeness. Recall that an (abstract) group $G$ is said to be 
{\em perfect} if $G$ coincides with its commutator subgroup $[G,G]$.
\begin{Fact} If $G$ is a central extension (as an abstract group) of a perfect group $P$, then $[G,G]$ is perfect.
\end{Fact}
\begin{proof} Let $N$ be the kernel of the surjective homomorphism $\pi:G\to P$. $N$ is central in $G$ by assumption. Let $H = [G,G]$. As $P$ is 
perfect $\pi(H) = P$, and for the same reason $\pi([H,H]) = P$. If by way of contradiction $H' = [H,H]$ is a proper subgroup of $H$, then, as $G = N\cdot H'$ we see that $[G,G] = H$ is contained in $H'$. Impossible. So $H$ is perfect.
\end{proof}

We now return to the $o$-minimal context.
\begin{Lemma} Suppose $R = Z(G)^{0}$. Let $S = [G,G]$. Then:
\newline
(i) $S$ is the unique maximal ind-definable semisimple subgroup of $G$.
\newline
(ii) $G = R\cdot S$.
\newline
(iii) $R\cap S$ is contained in $Z(S)$ and the latter is finitely generated.
\end{Lemma}
\begin{proof}
Let $P$ be the semisimple part of $G$, namely $P$ is definable semisimple and $\pi:G\to P$ is surjective with kernel $R = Z(G)^{0}$.
\newline
(i) We first prove that $S$ is ``ind-definable semisimple". Clearly $S$ is ind-definable. By Fact 1.5(ii), $P$ is perfect hence $\pi$ induces a surjective homomorphism $\pi|S:S \to P$. By Remark 2.6 it suffices to prove that $S$ is definably connected and $ker(\pi|S)$ is discrete. By Lemma 1.1 of \cite{CCII}, for each $n$, $[G,G]_{n}\cap Z(G)^{0}$ is finite, clearly showing that $ker(\pi|S)$ is discrete. If $S$ were not definably connected, let this be witnessed by the subgroup $S_{1}$ of $S$. As $P$ is definably connected, $\pi(S_{1}) = P$, hence $G = R\cdot S_{1}$, but then $S = [G,G]$ is contained in $S_{1}$, so $S = S_{1}$ (contradiction). 
\newline
We will now show maximality and uniqueness simultaneously by showing that {\em any} ind-definable semisimple subgroup $S_{1}$ of $G$ is contained in $S$. So let $S_{1}$ be such. By Lemma 2.7 $S_{1}$ is perfect, hence $S_{1} = [S_{1},S_{1}] \leq [G,G] = S$.
\newline
So (i) is proved. We now look at (ii) and (iii). As $P$ is perfect (1.5(ii)), $S$ maps onto $P$, so $G = R\cdot S$. As remarked earlier $[G,G]_{n}\cap R$ is finite for all $n$, whereby $R\cap S$ is discrete, so by Lemma 2.4 is central in $S$. Also by Fact 2.9, $Z(S)$ is finitely generated.

\end{proof}

\section{The almost linear case}
Assume to begin with that $G$ is {\em linear}, namely a definable subgroup of $GL_{n}(K)$. Let $\mathfrak{g}$ be its Lie algebra, a 
subalgebra of $\mathfrak{gl}_{n}(K)$. Let $\mathfrak{r}$ be $L(R)$ where remember that $R$ is the solvable radical of 
$G$. 

\begin{Lemma} ($G$ linear.) Let $S$ be a maximal ind-definable definably connected semisimple subgroup of $G$. Then $S$ is definable (in fact semialgebraic). Moreover  $G = R\cdot S$ and $R\cap S$ is contained in the (finite) centre of $S$.
\end{Lemma}
\begin{proof} Let $\mathfrak{s}$ be the Lie algebra of $S$ which is semisimple by Lemma 2.8. By \cite{Varadarajan} $\mathfrak{s}$ 
extends to a Levi factor $\mathfrak{s_{1}}$ of $\mathfrak{g}$ (i.e. a semisimple subalgebra such that $\mathfrak{g}$ is the semidirect product of $\mathfrak r$ and $\mathfrak s_1$). By Fact 1.5(v) there is a definable semisimple subgroup $S_{1}$ of $G$ such that $L(S_{1}) = \mathfrak{s_{1}}$. 
We will prove that $S\leq S_{1}$, so by maximality $S = S_{1}$ and is definable.
\newline
This is a slight adaptation of material from \cite{PPSI} to the present context. Consider the definable homogeneous space $X = G/S_{1}$. We have the natural action $\alpha:G\times X \to X$ of $G$ on $X$, which is differentiable (when $X$ is definably equipped with suitable differentiable structure). 
Let $a\in X$ be $S_{1}$ and let $f:G\to X$ be: $f(g) = g\cdot a$. By Theorem 2.19(ii) of \cite{PPSI},  $L(S_{1})$ is precisely the kernel of the differential $df_{e}$ 
of $f$ at the identity $e$ of $G$.  Consider the restriction $f_{1}$ of $f$ to $S$. As $L(S) = \mathfrak{s}$ is contained in $\mathfrak{s_{1}} = 
L(S_{1})$, we see that $(df_{1})_{e}$ is $0$. By  Theorem 2.19(i) of \cite{PPSI}, $(df_{1})_{h} = 0$ for all $h\in S$, so $f_{1}$ is ``locally constant" on $S$. It 
follows that $Fix(a) = \{h\in S:h(a) = a\}$ is a subgroup of $S$ which is ``locally" of finite index (as in Definition 2.3(ii)), hence $Fix(a) = S$ which means that 
$S\leq S_{1}$, as desired.

\vspace{2mm}
\noindent
For dimension reasons $G = R\cdot S$. Clearly $R\cap S$ is finite (for dimension reasons, or because it is solvable and normal in $S$), hence central in $S$ as $S$ is definably connected. 
\end{proof}

Now we want to prove conjugacy. 
\begin{Lemma} ($G$ linear.) Any two maximal ind-definable definably connected semisimple subgroups of $G$ are conjugate. 
\end{Lemma}
\begin{proof} Let $S, S_{1}$ be such. By Lemma 4.1 both $S, S_{1}$ are semialgebraic subgroups of $G\leq GL_{n}(K)$. By the proof of Lemma 3.1 in \cite{Pillay}, the (abstract) subgroup $H$ of $GL_{n}(K)$ generated by $S$ and $S_{1}$ is contained in some algebraic subgroup $H_{1}$ of $GL_{n}(K)$ such that moreover $H$ contains an open semialgebraic subset of $H_{1}$. 
For dimension reasons the definably (or equivalently semialgebraically) connected component of $H_{1}$ is contained in $G$. Hence we may assume that $G$ is already semialgebraic. (One could also get to this conclusion by using 4.1 of 
\cite{PPSIII} that there are semialgebraic  $G_{1} < G < G_{2}\leq GL_{n}(K)$ with $G_{1}$ normal in $G$ normal in $G_{2}$ and $G_{1}$ normal in $G_{2}$ such that $G_{2}/G_{1}$ is abelian.) 

We now make use of transfer to the reals together with the classical Levi theorem to conclude the proof.   Without loss of generality $G, S, S_{1}$ are defined by formulas $\phi(x,b)$, $\psi(x,b)$, $\psi_{1}(x,b)$ where these are formulas in the language of ordered fields with parameters witnessed by $b$. We may assume that these formulas include conditions on the parameters $b$ expressing that the group defined by $\phi(x,b)$ is definably connected and of the given dimension, also that the subgroups defined by $\psi(x,b)$, $\psi_{1}(x,b)$ are {\em maximal (semialgebraic) semisimple}. 
For example the family of definable abelian subgroups of a definable group is uniformly definable in terms of centralizers, so we can express that the subgroup defined by $\psi(x,b)$ is semisimple (definably connected with no  infinite normal abelian subgroup). We can also express maximality, by witnessing a solvable definable normal subgroup $R$ such that  $G$ is $R\cdot \psi(x,b)^{K}$. 
\newline
Let $\sigma$ be the sentence in the language of ordered fields expressing that for any choice $c$ of parameters, the subgroups of $\phi(x,c)^{K}$ defined by $\psi(x,c)$ and $\psi_{1}(x,c)$ are conjugate.
\newline
{\em Claim. $\sigma$ is true in the model $(\R,+,\cdot,<)$.}
\newline
{\em Proof of claim.}  Choose parameters $c$ from $\R$. Let $H, W, W_{1}$ be the groups (subgroups of $GL_{n}(\R)$) 
defined by the formulas $\phi(x,c)$, $\psi(x,c)$, $\psi_{1}(x,c)$. It is not hard to see that $W,W_{1}$ are maximal 
semisimple Lie subgroups of $H$ (which also happen to be closed) hence are conjugate in $H$ by 3.18.13 of 
\cite{Varadarajan}.
\newline
So the claim is proved, hence $\sigma$ is true in the structure $M$, whereby $S, S_{1}$ are conjugate in $G$.  
\end{proof}

Note that Lemmas 4.1 and 4.2 give Theorem 1.1 in the linear case. Let us now prove Theorem 1.2 in the linear case, by a slight extension of the proof of Lemma 4.1.

\begin{Lemma} ($G$ linear.) Let $\mathfrak{s}$ be a semisimple Lie subalgebra of $L(G)$. Then there is a {\em unique} ind-definable semisimple subgroup $S$ of $G$ such that $L(S) = \mathfrak{s}$. Moreover $S$ is definable.  
\end{Lemma}
\begin{proof} First let $S_{1}$ be a semialgebraic semisimple subgroup of $G$ with $\mathfrak{s} = L(S_{1})$. Let $S$ be another ind-definable semisimple subgroup of $G$ with $L(S) = {\mathfrak{s}}$. The proof of Lemma 4.1 shows that $S \leq S_{1}$. Let $P$ be a semisimple centreless definable group with $\pi:S\to P$ witnessing the semisimplicity of $S$ (according to Definition 2.5). By 1.5(vi) we may assume $P$ to be linear and semialgebraic. Let us now work inside the linear semialgebraic group $S_{1}\times P$. The graph of $\pi$, $W$ say, is clearly an ind-definable semisimple subgroup of
$S_{1}\times P$. Let $\mathfrak{w}$ be its Lie algebra, which is semisimple by 2.8. By 1.5(v), there is a semialgebraic semisimple $W_{1}\leq S_{1}\times P$ such that $L(W_{1}) = \mathfrak{w}$. Again as in the proof of 4.1 one sees that 
$W\leq W_{1}$. Note that $dim(W_{1}) = dim(S_{1}) = dim(P)$. So $W_{1}$ has finite cokernel. We will assume for simplicity that this cokernel is trivial whereby $W_{1}$ is he graph of a definable (semialgebraic) homomorphism $\pi_{1}$ from $S_{1}$ to $P$, which has to have finite kernel. Hence $ker(\pi)$ is also finite  from which it follows that $\pi$ is definable, as is $S$. Hence $S = S_{1}$. 
\end{proof} 

\vspace{2mm}
\noindent
We will say that $G$ is {\em almost linear} if for some finite central subgroup $N$ of $G$, $G/N$ is linear (i.e. there is a definable homomorphism from $G$ into some $GL_{n}(K)$ with finite central kernel). 

\begin{Lemma} Theorems 1.1 and 1.2 hold when $G$ is almost linear. Moreover in this case any ind-definable semisimple subgroup of $G$ is definable. 
\end{Lemma}
\begin{proof} 
Let $G/N$ be linear where $N$ is finite (central) and let $\pi:G\to G/N$ be the canonical surjective homomorphism. 
Note first that by Lemma 4.3 (and Lemma 2.8), any ind-definable semisimple subgroup of $G/N$ is definable. So if $S\leq G$ is ind-definable semisimple then $\pi(S)$ is definable, and so thus is $S$. This proves the moreover clause.
The rest easily follows from the previous lemmas. 
\end{proof}

\section{The general case}

\noindent
This is an easy consequence of the special cases in section 3 and 4 but we sketch the proofs nevertheless.
\newline
First:
\begin{proof}[Proof of Theorem 1.1.]
We construct the obvious  ind-definable semisimple subgroup $S$ of $G$,  observe the desired properties, then we prove its maximality and uniqueness up to conjugacy. \\

By Fact 1.5(iii), $G_{1} = G/Z(G)^{0}$ is almost linear. 
Let $\pi:G\to G_{1}$ be the canonical surjective homomorphism. Also let $R$ be the solvable radical of $G$ (which contains $Z(G)^{0}$) and then $R_{1} = \pi(R)$ will be the solvable radical of $G_{1}$.

By Lemma 4.4 let $S_{1}$ be a a definable Levi subgroup of $G_{1}$ and let $H = \pi^{-1}(S_{1})$, a definably connected definable subgroup of $G$. So $H$ is definably an extension of semisimple definable $S_{1}$ by the central subgroup $Z(G)^{0}$. So $Z(G)^{0}$ coincides with  $Z(H)^{0}$ and will  be the solvable radical of $H$. 
\newline
So Lemma 3.2  applies with $H$. Let $S = [H,H]$ which is in particular an ind-definable semisimple subgroup of $G$. As $S_{1}\cap R_{1}$ is finite, and $S\cap Z(G)^{0}$ is finitely generated, $S\cap R$ is finitely generated  (and normal, discrete hence central in $S$). 


It remains to prove, maximality and uniqueness up to conjugacy of $S$. 
Suppose $S_{2}$ is an ind-definable semisimple subgroup of $G$ containing $S$. Then $\pi(S_{2})$ is an ind-definable, semisimple, subgroup of $G_{1}$ containing $S_{1}$, so as $S_{1}$ was a Levi subgroup of $G_{1}$, 
$\pi(S_{2}) = S_{1}$. Hence $S_{2}$ is contained in $H$ and by Lemma 3.2 applied to $H$ equals $S$.

\vspace{2mm}
\noindent
Now let $S_{2}$ be another maximal ind-definable semisimple subgroup of $G$. Let $\pi(S_{2}) = S_{3}\leq G_{1}$. 
\newline
Clearly $S_{3}$ is an ind-definable semisimple subgroup of $G_{1}$, which is by Lemma 4.4 definable. Let $S_{4}$ be a maximal definable semisimple subgroup of $G_{1}$ containing $S_{3}$. Noting that  $G_{1} = R_{1}\cdot S_{4}$, and as (by Lemma 2.7), $S_{2}$ is perfect, we see by Lemma 3.2, that  $[\pi^{-1}(S_{4}),\pi^{-1}(S_{4})]$ is ind-definable semisimple and contains $S_{2}$ hence by maximality of $S_{2}$ we have equality, and $S_{3} = S_{4}$ and moreover $S_{2} = [\pi^{-1}(S_{3}),\pi^{-1}(S_{3})]$. By Lemma 4.4 again $S_{3}$ is conjugate in $G_{1}$ to $S_{1}$ by $g_{1}$ say.
Hence (as $ker(\pi)$ is central in $G$) for any lift $g$ of $g_{1}$ to a point of $G$, $S_{2} = [\pi^{-1}(S_{3}),\pi^{-1}(S_{3})]$ is conjugate via $g$ to $S = [\pi^{-1}(S_{1}),\pi^{-1}(S_{1})]$, and we have proved conjugacy.

This completes the proof of Theorem 1.1.
\end{proof}

\begin{proof}[Proof of Theorem 1.2.] 
Again we let $\pi:G\to G/Z(G)^{0} = G_{1}$ be the canonical surjective homomorphism. If $\mathfrak{s}$ is semisimple and $\mathfrak{s}_{1} = \pi(\mathfrak{s})$ (namely the image of $\mathfrak{s}$ under the differential of $\pi$ at the identity), then $\mathfrak{s}_{1}$ is also semsimple and so by Lemma 4.4 is the Lie algebra of a unique definable semisimple subgroup $S_{1}$ of $G_{1}$. Let $H = \pi^{-1}(S_{1})$. Then by Lemma 3.2 applied to $H$ we see that $\mathfrak{s}$ is the unique Levi factor of $L(H)$ and $S = [H,H]$ is the unique ind-definable semisimple subgroup of $H$ with $L(S) = \mathfrak{s}$. 
\newline
If $S_{2}$ is another ind-definable semisimple subgroup of $G$ with $L(S_{2}) = \mathfrak{s}$, then by Lemma 4.4, $\pi(S_{2}) = S_{1}$, and by perfectness of $S_{2}$ and definable connectedness of $[H,H] = S$ we see that $S_{2} = S$. 

Theorem 1.2 is proved.
\end{proof}

\noindent
Finally we mention cases when some (any) ind-definable Levi subgroup of $G$ is definable. $G$ remains a definably connected group definable in $M$. 
\begin{Proposition} Suppose either of the following hold:
\newline
(i) $G$ is affine Nash,
\newline
(ii) $G/N$ is linear for some finite central $N$,
\newline
(iii) the semisimple part $P$ of $G$ has finite $o$-minimal fundamental group,
\newline
(iv) the semisimple part $P$ of $G$ is definably compact.
\newline
THEN any ind-definable Levi subgroup $S$ of $G$ is definable (whereby $G = R\cdot S$ with $R\cap S$ finite).
\end{Proposition}
\begin{proof} (i) $G$ being affine Nash means that $G$ is definable in the $RCF$ language and with its unique structure as a Nash manifold has a Nash embedding in some $K^{n}$. See \cite{Hrushovski-Pillay}. In fact in the latter paper it is proved that $G$ is a finite cover of an ``algebraic group" (namely of $H(K)^{0}$ where $H$ is an algebraic group defined over $K$). Remark 2.9 of \cite{CCI} says that $H(K)^{0}$ has a definable Levi subgroup, and this lifts to $G$.
\newline
(ii) This is already part of Lemma 4.4 above.
\newline
(iii) The previous material shows that it suffices to look at central extensions $G$ of a semisimple group $P$, in which case by Lemma 3.2, $S = [G,G]$ is the ind-definable Levi subgroup. The induced surjective homomorphism $S\to P$ is, by Remark 2.10, a quotient of the $o$-minimal universal cover of $P$, so if $P$ had finite $o$-minimal fundamental group, then the kernel of $S\to P$ is finite whereby easily $S$ is definable.
\newline
(iv) If $P$ is definably compact then it has finite $o$-minimal fundamental group (as this corresponds to the usual fundamental group of the associated compact semisimple Lie group discussed at the end of section 3), so we can use (iii). But this case also follows from the results of \cite{HPP}.
\end{proof}

\end{document}